\documentclass[a4paper,12pt%
]{amsart}
\pdfoutput=1
\usepackage{amsmath, amssymb, amsthm, amscd}
\usepackage{graphicx}
\usepackage{enumerate}
\usepackage[all,cmtip]{xy}

\newcounter{theorem}

\newtheorem{lemma}[theorem]{Lemma}
\newtheorem{prop}[theorem]{Proposition}
\newtheorem{cor}[theorem]{Corollary}
\newtheorem{defn}[theorem]{Definition}
\newtheorem{notation}[theorem]{Notation}

\theoremstyle{remark}
\newtheorem*{remark*}{Remark}

\newtheorem{example}[theorem]{Example}

\numberwithin{equation}{section}
\numberwithin{theorem}{section}

\DeclareMathOperator\type{type}
\DeclareMathOperator\inner{inner}
\DeclareMathOperator\outr{outer}
\DeclareMathOperator\adj{adj}

\newcommand{\defemph}{\emph}

\newcommand{\bigmatrix}[1]{%
\left[ \begin{array}{c|rrr}
#1 & 0 & \cdots & 0 \\
\hline
#1 &  & & \\
\vdots &  & B' & \\
#1 &  & &  \\
\hline
#1 &  & &  \\
\vdots & & -B' & \\
#1 &  & &  
\end{array}\right]}%
\newcommand{\R}{\mathbb{R}}
\newcommand{\Q}{\mathbb{Q}}
\newcommand{\Z}{\mathbb{Z}}
\newcommand{\N}{\mathbb{N}}

\newcommand{\state}{\tau}
\newcommand{\basisb}{{\mathfrak B}}

\renewcommand{\setminus}{\backslash}

\title[Ultrasimplicial groups with unique state]{The ultrasimplicial property for simple dimension groups with unique state, the image of which has rank one}

\author{Gregory R. Maloney}
\address{Newcastle University}
\thanks{This work was supported by the Fields Institute during a research visit.}
\date{\today}

\begin{document}

\begin{abstract}
Let $G$ be an ordered group that is a direct sum of a rank-one torsion-free abelian group and a finite-rank torsion-free abelian group, with order structure arising from the natural order on the first summand.  
A necessary condition and a sufficient condition are given for $G$ to have an ordered-group inductive limit representation using injective maps.  
\end{abstract}

\maketitle

\section{Introduction}\label{SEC:intro}

\begin{defn}\label{DEF:dimension-group}
A countable ordered Abelian group is called a \defemph{dimension group} if it is isomorphic to the inductive limit of a sequence of simplicial groups (direct sums of copies of $\Z$) in the category of ordered Abelian groups.  
\end{defn}

Let us always assume that ordered groups $G$ are directed (if $g_1,g_2\in G$, then there exists $h\in G$ with $g_1,g_2\leq h$) and unperforated (if $ng \geq 0$ for some $n\in \N$, then $g\geq 0$).  
These properties automatically hold for dimension groups, and are frequent assumptions in the literature.  
An excellent reference on ordered abelian groups is \cite{G:book}.  

\begin{defn}\label{DEF:ultrasimplicial}
A dimension group is called \defemph{ultrasimplicial} if it is isomorphic to such an inductive limit in which the maps are injective.  
\end{defn}

It is easy to construct examples of ultrasimplicial groups.  

\begin{example}\label{EX:dyadic}
Let $A$ be an $n\times n$ matrix of non-negative integers.  
Then $A$ represents a morphism from $\Z^n$ to $\Z^n$ in the category of ordered abelian groups, and the inductive limit of the diagram

\centerline{
\xymatrix{
\Z^n \ar@{>}^{A}[r] & \Z^n \ar@{>}^{A}[r] & \Z^n \ar@{>}^{A}[r] & \Z^n \ar@{>}^{A}[r] & \Z^n \ar@{>}^{A}[r] & \cdots \\
}
}
is a dimension group.  
If $A$ has full rank, then it is ultrasimplicial.  
The dyadic rationals $\Z\big[\frac{1}{2}\big]$ arise in this way if we choose $A:\Z\to\Z$ to be multiplication by $2$.  
\end{example}

\begin{example}\cite[Example 2.7]{E:example}.\label{EX:elliott}
The inductive limit

\centerline{
\xymatrix{
\Z^3 \ar@{>}^{A}[r] & \Z^3 \ar@{>}^{A}[r] & \Z^3 \ar@{>}^{A}[r] & \Z^3 \ar@{>}^{A}[r] & \Z^3 \ar@{>}^{A}[r] & \cdots \\
}
}
where the map $A$ is given by left multiplication by the matrix
\[
A = \left[ \begin{array}{ccc}
1 & 1 & 1\\
2 & 1 & 0\\
0 & 1 & 2
\end{array}\right]
\]
is a dimension group that is not ultrasimplicial.  
It is isomorphic to $G := \Z\left[ \frac{1}{3}\right] \oplus \Z$ with order arising from the positive cone $G^+ := \{ (a,b)\in G \ | \ a > 0\} \cup \{ 0\}$.  
\end{example}

The matrix $A$ in Example \ref{EX:elliott} has rank two, so it obviously does not give rise to an injective map.  
Nevertheless, more work is required to show that $(G,G^+)$ is not ultrasimplicial, because it might have some other inductive limit representation in which the maps are injective.  
To see that this does not happen, it is necessary to use the following intrinsic characterization of ultrasimplicial groups.  

\begin{prop}\cite[Proposition 1]{H:ultrasimplicial}.\label{PROP:intrinsic}
A countable ordered Abelian group $G$ is ultrasimplicial if and only if, for all finite $F\subseteq G^+$, there exists a finite independent set $S\subseteq G^+$ such that, for all $f\in F$, $f$ can be written in the form
\[
f = \sum n_is_i
\]
where $n_i\in \Z_{\geq 0}$ and $s_i\in S$.
\end{prop}
We can see that the group $G$ of Example \ref{EX:elliott} is not ultrasimplicial by applying Proposition \ref{PROP:intrinsic} to the elements $(1,0),(1,1),$ and $(1,-1)$.  

The goal here is to study the ultrasimplicial question for other finite-rank dimension groups like this one.  
(Dimension groups are torsion-free, so let us use the word \defemph{rank} in this context to refer to the torsion-free rank of an abelian group.)  
Let us describe how this work fits in the literature on dimension groups.  

\section{Background}\label{SEC:background}

In \cite{E:example}, Elliott introduced the notion of an ultrasimplicial group and showed that all totally ordered abelian groups are ultrasimplicial, and then offered Example \ref{EX:elliott} to show that not all dimension groups are ultrasimplicial.  
Effros and Shen recognized that Elliott's proof that totally ordered groups are ultrasimplicial relied essentially on a continued fraction algorithm, and in \cite{ES:rank-two} they showed how to use the ordinary continued fraction algorithm to construct an explicit inductive limit realizing the ultrasimplicial property for a simple dimension group with two generators and a unique state.  
(A \defemph{state} is a normalized positive homomorphism of an ordered group into the real numbers.)  
In \cite{ES:multidimensional} they used a multi-dimensional continued fraction algorithm to extend this construction from two generators to finitely many generators under the assumption of total order.  

Then in \cite{R:algorithm} Riedel showed, using another continued fraction algorithm, that any simple finitely-generated dimension group with a unique state, the image of which has rank bigger than one, is ultrasimplicial.  
The hypothesis of finite generation in Riedel's theorem can easily be relaxed to finite rank without significantly altering the proof; therefore any simple finite-rank dimension group with unique state, the image of which has rank bigger than one, is ultrasimplicial.  

There is an obvious unanswered question here: what if the image of the unique state has rank equal to one?  
Both the dyadic rationals (Example \ref{EX:dyadic}) and Elliott's example (Example \ref{EX:elliott}) fall into this class, and the former is ultrasimplicial while the latter is not.  

This work provides a partial answer to that question.  
Specifically, let us consider simple finite-rank dimension groups that have unique state, the image of which has rank one, and that satisfy the additional condition that the exact sequence associated to the state splits.  
This work provides a sufficient condition (Proposition \ref{PROP:sufficient}) and a necessary condition (Proposition \ref{PROP:necessary}) for such groups to be ultrasimplicial.  
These conditions involve the divisibility properties of the image and kernel of the unique state.  

The present work takes Riedel's result in \cite{R:algorithm} as a starting point, but there have been other developments since then.  

After proving his result for the case in which there is a unique state with image rank greater than one, Riedel showed in \cite{R:counterexample} that this cannot be extended to simple dimension groups with more than one state; specifically, for any $n>1$ there is a simple dimension group with $n+1$ generators and $n$ extreme states that is not ultrasimplicial.  
There is no overlap between the counterexamples constructed in \cite{R:counterexample} and the groups discussed in the present work, because in the present work there is only one state.  

In \cite{M:lattice-ordered}, Marra showed that every lattice-ordered abelian group is ultrasimplicial.  
The groups discussed here are lattice-ordered only if the kernel of the unique state is trivial, so there is only trivial overlap between this work and Marra's result.  

In \cite{T:vector-spaces}, Tikuisis showed that a finite-dimensional ordered vector space over a subfield of the real numbers is ultrasimplicial.  
There is some overlap between Tikuisis's results and the results here; this overlap consists of the simple finite-dimensional ordered rational vector spaces with unique state.  
But this is only a small subset of the vector spaces discussed by Tikuisis, who does not require simplicity or a unique state.  
It is likewise only a small subset of the groups discussed in this work, which deals with groups that do not necessarily have a vector space structure.  

In \cite{H:almost-ultrasimplicial}, Handelman discussed a broad class of dimension groups, including the ones treated here.  
He called a dimension group \defemph{co-rank one ultrasimplicial} if it can be represented as an ordered-group inductive limit 

\centerline{
\xymatrix{
\Z^n \ar@{>}^{A_1}[r] & \Z^n \ar@{>}^{A_2}[r] & \Z^n \ar@{>}^{A_3}[r] & \Z^n \ar@{>}^{A_4}[r] & \Z^n \ar@{>}^{A_5}[r] & \cdots 
}
}

such that the kernel of any telescoping $A_mA_{m-1}\cdots A_{n+1}A_n$, with $m>n$, has rank at most one.  
He then showed that every simple dimension group with unique state is co-rank one ultrasimplicial by constructing an inductive limit representation using square matrices of the appropriate size with equal column sums.  

The results presented here are different from Handelman's in that they address the question of ultrasimpliciality, not co-rank one ultrasimpliciality, albeit for a smaller class of dimension groups.  
Some of the techniques used here---specifically in the proof of Proposition \ref{PROP:sufficient}---are inspired by techniques appearing in \cite{H:almost-ultrasimplicial}.  

\section{Notation and definitions}\label{SEC:notation}

Let $G$ be a simple finite-rank torsion-free ordered abelian group with unique state $\state : G\to \R$.  
When the state $\state$ is unique, to say that $G$ is simple means that $g\in G^+$ if and only if $g = 0$ or $\state(g)$ is a strictly positive real number.  
Let $K$ denote the kernel of $\state$ and let $H\subset\R$ denote its image.  
$G$ has finite rank, so $K$ and $H$ do as well.  

Suppose that $H$ has rank one, and that the exact sequence

\centerline{
\xymatrix{
0 \ar@{>}[r] & K \ar@{>}[r] & G \ar@{>}^{\state}[r] & H \ar@{>}[r] & 0 
}
}
splits.  
This is the setting of interest in this work.  

Let us give a more precise description of such groups $G$, and at the same time let us fix notation to be used in what follows.  

\begin{notation}\label{NOT:g}
Let $K$ be a finite-rank torsion-free abelian group, and let $H$ be a rank-one torsion-free abelian group.  
$H$ is isomorphic to a subgroup of $\Q$; let $H^+ = H\cap \Q^+$.  
(Different choices of embeddings into $\Q$ yield two such orderings; pick one of them.)  

Let $G$ denote the ordered group that is equal to $H\oplus K$ as a group, and that has order arising from the positive cone
\[
G^+ = \{ (h,k)\in H\oplus K \ | \ h > 0 \} \cup \{ 0\}.
\]
\end{notation}

Every group $G$ of this form is a simple finite-rank torsion-free abelian group with unique state $\state$, the image of which has rank one.  
Moreover, every group satisfying all of those conditions, plus the additional condition that the exact sequence associated to $\state$ split, is isomorphic to some group $G$ of this form.  
Therefore let us henceforth consider only groups $G$ of this form.  

The main question considered here is: which of these groups $G$ are ultrasimplicial?  
If the group $K$ is trivial, then $G$ is totally ordered, in which case it is ultrasimplicial by the result of Elliott \cite{E:example}.  
Therefore let us only consider ordered groups $G$ for which $K$ is non-trivial.   
The results presented here are expressed in terms of certain invariants of torsion-free abelian groups, called type invariants.  
See \cite{A:book} for an exposition of the subject.  

\begin{defn}\label{DEF:height}
Let $A$ be a torsion-free abelian group, let $a\in A$, and let $p\in \Z$ be a prime number.  
The \defemph{$p$-height of $a$ in $A$}, denoted $h_p^A(a)$, is defined to be $n$ if there is an $n\geq 0$ such that $a\in p^nA\setminus p^{n+1}A$, and $\infty$ if no such $n$ exists.  

The \defemph{height of $a$ in $A$} is defined to be the sequence $(h_p^A(a))_{p\in \Pi}$ indexed by the set $\Pi$ of all primes in $\Z$.  
\end{defn}

The set of all height sequences is an ordered set under the relation $\alpha\leq \beta$ if $\alpha_p\leq\beta_p$ for each $p\in\Pi$.  
Any two elements have a supremum and infimum: $\alpha\vee\beta = (\max \{\alpha_p,\beta_p\})_{p\in\Pi}$ and $\alpha\wedge\beta = (\min\{ \alpha_p,\beta_p\})_{p\in\Pi}$, so in fact this is a lattice-ordered set.  

\begin{defn}\label{DEF:type}
Two height sequences $(\alpha_p)$ and $(\beta_p)$ are called \defemph{equivalent} if they differ from each other in a finite number of positions, and for any $p$ if $\alpha_p=\infty$ or $\beta_p = \infty$, then $\alpha_p = \beta_p$.  
It is easy to check that this is an equivalence relation.  

The \defemph{type of $a$ in $A$}, denoted $\type_A(a)$, is the equivalence class of $(h_p^A(a))_{p\in\Pi}$ under this relation.  
\end{defn}

The set of all types is also a lattice-ordered set under the relation induced by the order on the set of all height sequences.  

Note that, in a rank-one group $A$, the heights of any two non-zero elements are equivalent, and so all non-zero elements of $A$ will have the same type.  
Denote this common type by $\type(A)$.  

If $A$ has rank greater than one, then the set of types of elements of $A$ need not be a singleton.  
In such cases, more refined definitions of type have been introduced.  
These definitions make reference to the notion of a \defemph{pure} subgroup.  

\begin{defn}\label{DEF:pure-rank-one}
Let $A$ be a torsion-free abelian group.  
A subgroup $B$ of $A$ is called \defemph{pure} if $na\in B$ implies $a\in B$ for $a\in A,n\in \N$.  

Given a subset $S$ of $A$, the \defemph{pure subgroup generated by $S$}, denoted by $\langle S\rangle_*$, is the subgroup defined as follows.  
\begin{align*}
\langle S\rangle_* & = \{ a\in A \ | \ na\in \langle S\rangle \text{ for some } n\in\Z\}.
\end{align*}
\end{defn}

\begin{defn}\label{DEF:innertype}
Let $A$ be a finite-rank torsion-free abelian group, and let $x_1,\ldots,x_n$ be a maximal independent subset of $A$.  
Then the \defemph{inner type of $A$}, denoted $\inner \type (A)$, is the infimum $\bigwedge_{i=1}^n \type_A(x_i) = \bigwedge_{i=1}^n \type \langle x_i\rangle_*$.  
It is not hard to check that the inner type of $A$ does not depend on the particular choice of maximal independent subset $\{ x_1,\ldots,x_n\}$ \cite[Proposition 1.7]{A:book}.  
\end{defn}

The inner type of $K$ plays a role in Proposition \ref{PROP:sufficient}, which gives a sufficient condition for $G$ to be ultrasimplicial.  
A quantity called the \defemph{outer type} of $K$ plays a similar role in Proposition \ref{PROP:necessary}, which gives a necessary condition for $G$ to be ultrasimplicial.  

\begin{defn}\label{DEF:outertype}
Let $A$ be a finite-rank torsion-free abelian group, and let $x_1,\ldots,x_n$ be a maximal independent subset of $A$.  
Let $Y_i = \langle x_1,\ldots,$ $x_{i-1},$ $x_{i+1},\ldots,x_n\rangle_*$.  
Then $A/Y_i$ is torsion-free of rank-one.  
The \defemph{outer type of $A$}, denoted $\outr \type A$, is the supremum $\bigvee_{i=1}^n \type A/Y_i$.  
It is not hard to check that the outer type of $A$ does not depend on the particular choice of maximal independent subset $\{ x_1,\ldots,x_n\}$ \cite[Proposition 1.8]{A:book}.  
\end{defn}

\section{Results}\label{SEC:results}

The main results rely on the following lemma, which is easy to prove using induction and the Euclidean algorithm.  

\begin{lemma}\label{LEM:euclidean}
If $A$ is a finite-rank torsion-free abelian group and $a_1',\ldots,$ $a_l'\in A$, then there exist independent $a_1,\ldots,a_n\in A$ such that $\langle a_1',\ldots,a_l'\rangle$ $= \langle a_1,\ldots, a_n\rangle$.  
\end{lemma}

The first result gives a sufficient condition for $G$ to be ultrasimplicial.  

\begin{prop}\label{PROP:sufficient}
If $\inner\type K \wedge \type H \neq \type \Z$, then $G$ is ultrasimplicial.  
\end{prop}

\begin{proof}
Use Proposition \ref{PROP:intrinsic}.  
Pick $(h_1',k_1'),\ldots,(h_l',k_l')\in G^+$, and let us find independent positive elements that generate these as non-negative integer combinations.  
By the definition of $G^+$, we must have $h_1',\ldots,h_l'\in H^+$.  

By Lemma \ref{LEM:euclidean} there is an element $h\in H$ with $\langle h_1',\ldots,h_l'\rangle = \langle h\rangle$.  
Since $H$ is directed and has rank one, we must have $h\in H^+$ or $-h\in H^+$, so, by taking $-h$ if necessary, we may suppose $h\in H^+$.  
Likewise there are $n$ elements $k_1,\ldots,k_n\in K$ such that $\langle k_1',\ldots,k_l'\rangle$ $ = \langle k_1,\ldots,k_n\rangle$.  

If $\inner \type K\wedge \type H \neq \type\Z$, then there is an infinite sequence of possibly repeated primes $p_1,p_2,\ldots$ such that, for all $N\in\N$, the elements $h,k_1,\ldots,k_n$ are all divisible by $P_N := \prod_{i=1}^Np_i$---that is, $h\in P_NH$ and $k_1,\ldots,k_n\in P_N K$.  
Because $G$ is torsion free, non-zero elements that are divisible by $P_N$ are uniquely divisible by $P_N$; that is, if $P_Ng_1=g_0$ and $P_Ng_2=g_0$, then $g_1 = g_2$.  
Therefore we may speak without ambiguitiy of the elements $\frac{1}{P_N}h,$ $\frac{1}{P_N}k_1,\ldots, \frac{1}{P_N}k_n$.  
We can use these elements, with $N$ sufficiently large, to produce positive independent elements satisfying the requirements of Proposition \ref{PROP:intrinsic}.  

Let $M_1,\ldots,M_n,M_1',\ldots,M_n'$ be positive integers, to be determined later.  
Define elements $m_0,m_1,\ldots,m_n\in K$ by $m_i := M_ik_i$ for $1\leq i \leq n$ and $m_0 := -M_1'k_1-\cdots -M_n'k_n$.  
Then the positive elements $(h,m_0),(h,m_1),\ldots,(h,m_n)$ generate a subgroup of $\langle (h,0),(0,k_1),\ldots,$ $(0,k_n)\rangle$ of index $|\det(M)|$, where
\begin{align}
M & := \left[ \begin{array}{rrrrr}
1 & 1 & 1 & \cdots & 1 \\
-M_1' & M_1 & 0 & & \\
-M_2' & 0 & M_2 & & \\
\vdots & & & \ddots & \\
-M_n' & & & & M_n
\end{array}\right].\label{EQ:M}
\end{align}

Then the following statements, both proved in lemmas below, are sufficient to prove the claim.
\begin{enumerate}
\item  Suppose that, for all $i,j\leq n$, $\frac{1}{2}M_i' \leq M_j'\leq 2M_i'$.  
Pick $(h',k')\in G^+$.  
If $M_1,\ldots,M_n,M_1',\ldots,M_n'$ are sufficiently large, then we can express $(h',k')$ as a non-negative rational combination of $(h,m_0),\ldots,(h,m_n)$; that is,
\begin{align*}
q(h',k') & = q_0(h,m_0) + \cdots + q_n(h,m_n)
\end{align*}
for some non-negative integers $q,q_0,\ldots,q_n$.  
The proof of this statement appears in Lemma \ref{LEM:matrix-inverse}.  
\item  Let $L\in\N$ be arbitrary.  
Then if $N$ is sufficiently large, we can choose $M_1,\ldots,M_n,M_1',\ldots,M_n'$, all larger than $L$, with $\frac{1}{2}M_j'\leq M_i'\leq 2M_j'$ for all $i,j\leq n$, such that $|\det(M)| = P_N$.  
The proof of this statement appears in Lemma \ref{LEM:factorization}.  
\end{enumerate}
The result follows from this because, by statement 1, we can find some positive integer $q$ such that all elements $q(h_1',k_1'),\ldots,q(h_l',k_l')$ can be expressed as non-negative integer combinations of $(h,m_0),\ldots,(h,m_n)$.  
At the same time, by statement 2, we can arrange it so that $P_N$ is the index of $\langle (h,m_0),\ldots,(h,m_n)\rangle$ in $\langle (h,0),(0,k_1),\ldots,(0,k_n)\rangle$.  
Then $q$ divides $P_N$, so $(h_1',k_1'),\ldots,(h_l',k_l')$ are non-negative integer combinations of $\frac{1}{P_N}(h,m_0),\ldots,\frac{1}{P_N}(h,m_n)\in G^+$.  
\end{proof}

Let $M((M_i')_{i=1}^n;(M_i)_{i=1}^n)$ denote the matrix $M$ in Equation \ref{EQ:M}, and let $D((M_i')_{i=1}^n;(M_i)_{i=1}^n)$ denote its determinant.  
Then by expanding along the first row, we obtain
\begin{align*}\label{EQ:det}
D((M_i')_{i=1}^n;(M_i)_{i=1}^n) & = \prod_{i=1}^nM_i + \sum_{i=1}^n M_i'\prod_{j\neq i}M_j > 0.
\end{align*}

Let us now prove the two lemmas that were used in the proof of Proposition \ref{PROP:sufficient}.  

\begin{lemma}\label{LEM:matrix-inverse}
Let $M_1,\ldots,M_n,M_1',\ldots,M_n'$ be positive integers with the property that $\frac{1}{2}M_i' \leq M_j'\leq 2M_i'$ for all $i,j\leq n$.  
Let $B = (b_0,b_1,\ldots,b_n)^t$ be a column vector of $n+1$ integers with $b_0>0$ and let $X = (x_0,x_1,\ldots,$ $x_n)^t$ be the solution to the matrix equation $MX = B,$ where $M$ is defined in Equation \ref{EQ:M}.  
Then there exists $L\in\N$ such that, if $M_i,M_i'>L$ for all $i\leq n$, then all entries of $X$ are non-negative.  
\end{lemma}
\begin{proof}
For brevity, let us denote $M((M_i')_{i=1}^n;(M_i)_{i=1}^n)$ by $M$ and its determinant by $D$.  
Then we have
\begin{align*}
X & = D^{-1}\adj (M) B,
\end{align*}
where $\adj(M)$ denotes the adjugate of $M$---that is, the transpose of its matrix of cofactors.  

Let $M(i,j)$ denote the $(i,j)$th cofactor of $M((M_i')_{i=1}^n;(M_i)_{i=1}^n)$, where the indices start at $0$.  
The cofactors of the first column are
\begin{equation}
M(0,0)  = \prod_{j=1}^n M_j\label{EQ:top-corner}
\end{equation}
and
\begin{equation}
M(i,0)  = \prod_{j\neq i}M_j,\label{EQ:left-column}
\end{equation}
and the cofactors of the first row are
\begin{equation}
M(0,i)  = M_i'\prod_{j\neq i}M_j.\label{EQ:top-row}
\end{equation}
The cofactors along the diagonal are determinants of the same type as $D$; that is, 
\begin{equation}
M(i,i)  = D((M_j')_{j\leq n, j\neq i};(M_j)_{j\leq n, j\neq i}), \quad i\geq 1.\label{EQ:diagonal}
\end{equation}
The remaining cofactors are
\begin{equation}
M(i,j)  = -M_j'\prod_{l\neq i,j}M_l.\label{EQ:the-rest}
\end{equation}

Let $b = \max_i |b_i|/b_0$, and suppose that each $M_i',M_i$ is greater than $L > 2n^2b$.  
Then using Equations \ref{EQ:top-corner} and \ref{EQ:left-column}, we see that 
\begin{align*}
M(0,0) & = \prod_{j=1}^n M_j \geq L\prod_{j\neq i}M_j = L |M(i,0)| > 2n^2b|M(i,0)|.
\end{align*}

Using Equations \ref{EQ:top-row} and \ref{EQ:the-rest}, we see that
\begin{align*}
M(0,i) & = M_i'\prod_{l\neq i} M_l \geq LM_i'\prod_{l\neq i,j}M_l \geq \frac{1}{2}LM_j' \prod_{l\neq i,j}M_l = \frac{1}{2}L|M(i,j)|
\end{align*}
when $j\neq i$.  

When $1\leq j = i$ the calculation, using Equation \ref{EQ:diagonal}, is more involved.  
\begin{align*}
M(i,i) & = D((M_j')_{j\leq n, j\neq i};(M_j)_{j\leq n, j\neq i}) \\
& = \prod_{j\neq i}M_j + \sum_{j\neq i}M_j'\prod_{l\neq i,j}M_l \\
& \leq \left( \prod_{j\neq i}M_j\right)\left(1 + \frac{1}{L}\sum_{j\neq i}M_j'\right) \\
& = \left(  \frac{M(0,i)}{M_i'}\right)\left(1 + \frac{1}{L}\sum_{j\neq i}M_j'\right) \\
& \leq \left( \frac{M(0,i)}{M_i'}\right) \left(1 + \frac{1}{L}\sum_{j\neq i}2M_i'\right) \\
& =  \frac{M(0,i)}{M_i'}\left(1 + \frac{2}{L}(n-1)M_i'\right) \\
& =  M(0,i)\left(\frac{1}{M_i'} + \frac{2}{L}(n-1)\right) \\
& \leq  M(0,i)\left(\frac{2}{L} + \frac{2}{L}(n-1)\right) \\
& = \frac{2n}{L}M(0,i) < \frac{1}{nb}M(0,i).
\end{align*}

Thus, in all cases, $M(0,i) > nb|M(j,i)|$ for $j\geq 1$, so
\begin{align*}
Dx_i & = b_0M(0,i) + b_1M(1,i) + \cdots + b_nM(n,i) \\
& \geq b_0M(0,i) - n\max_j |b_j|\max_{j\geq 1} |M(j,i)| \\
& > b_0(n\max_j |b_j|/b_0)\max_{j\geq 1}|M(j,i)| - n\max_j|b_j|\max_{j\geq 1}|M(j,i)| = 0.
\end{align*}
\end{proof}

\begin{lemma}\label{LEM:factorization}
Let $p_1,p_2,\ldots$ be an infinite sequence of possibly repeated primes, and let $P_N = \prod_{i=1}^N p_i$.  
Let $L\in\N$ be arbitrary.  
Then there exists $N\in\N$ such that it is possible to express $P_N$ in the form
\begin{equation*}
P_N = \prod_{i=1}^n M_i + \sum_{i=1}^nM_i'\prod_{j\neq i}M_j,
\end{equation*}
where $M_i,M_i'\in\N$ are all greater than $L$, and $\frac{1}{2}M_i'\leq M_j'\leq 2 M_i'$ for all $i,j\leq n$.  
\end{lemma}
\begin{proof}
We may suppose that $L>3$.  

If $N$ is very large, then $P_N$ has many large factors.  
In particular we can choose $N$ large enough that $P_N$ can be written as a product of $n+1$ positive factors $f_0\cdot f_1\cdot \cdots\cdot f_n$, each of which is larger than $L$, and the first of which, $f_0$, is larger than $\sum_{i=1}^n\prod_{1\leq j\leq n,j\neq i}f_j$.  

Let $\overline{f}_i = \prod_{1\leq j\leq n,j\neq i}f_j$, and let $\overline{f} = \sum_{i=1}^n \overline{f}_i$.  
Then $f_0 > \overline{f}$, and we can write 
\begin{equation*}
f_0 = 1+\sum_{i=1}^n F_i,
\end{equation*}
where
\begin{equation*}
F_i = \left\lfloor \frac{f_0\overline{f}_i}{\overline{f}}\right\rfloor + \epsilon_i,
\end{equation*}
with $\epsilon_i = 0$ or $1$.  
(Here $\lfloor\cdot\rfloor$ denotes the floor function, or greatest integer function.)  
Perhaps one $\epsilon_i$ will be $-1$, to account for the fact that $\sum_{i=1}^nF_i = f_0-1$, instead of $f_0$.  
So in particular $|F_i-\frac{\overline{f}_if_0}{\overline{f}}| \leq 1$.  

Then
\begin{align*}
P_N & = \prod_{i=0}^nf_i \\
& = \prod_{i=1}^n f_i + (f_0-1)\prod_{i=1}^nf_i \\
& = \prod_{i=1}^nf_i + \sum_{i=1}^nF_i\prod_{j=1}^nf_j \\
& = \prod_{i=1}^nf_i + \sum_{i=1}^nf_i'\overline{f}_i,
\end{align*}
where $f_i' = F_if_i$.  

Taking $M_i = f_i$ and $M_i' = f_i'$, this is the required form for $P_N$; now we need only check that $\frac{1}{2}f_i'\leq f_j'\leq 2f_i'$ for all $1\leq i,j\leq n$.  

But this is not difficult.  
This condition is automatically satisfied if $n = 1$, and if $n > 1$ then
\begin{align*}
|F_i-\frac{\overline{f}_if_0}{\overline{f}}| & < 1 \\
|F_if_i-\frac{\overline{f}_if_if_0}{\overline{f}}| & < f_i \\
|f_i'-\frac{\prod_{j =0}^nf_j}{\overline{f}}| & < f_i.
\end{align*}

Let $T = P_N/\overline{f} = (\prod_{j=0}^nf_j)/\overline{f} > \prod_{j = 1}^n f_j$.  Then $|f_i'-T|<f_i$, and
\begin{align*}
2f_i'-f_j' & = 2(f_i'-T) + T - (f_j'-T) \\
& \geq -2f_i + T - f_j \\
& \geq T - 3\max_{l\geq 1}f_l \\
& > \prod_{l = 1}^n f_l - 3\max_{l\geq 1}f_l >0
\end{align*}
because each $f_l>L > 3$.  
\end{proof}

Proposition \ref{PROP:necessary}, below, says that, if $\outr\type K \wedge \type H = \type \Z$, then $G$ is not ultrasimplicial.  
The hypothesis that $\outr\type K \wedge \type H = \type \Z$ has one particular implication that is used in the proof of Proposition \ref{PROP:necessary}; this implication is proved here in Lemma \ref{LEM:outertype}.  

\begin{lemma}\label{LEM:outertype}
If $\outr\type K\wedge \type H = \type \Z$, then for any maximal independent subset $\{ k_1',\ldots,k_n'\}$ of $K$ we can find non-zero $h'\in H$ such that, if 
\begin{align*}
g := c_1k_1'+\cdots + c_nk_n'
\end{align*}
is an integer combination of $k_1',\ldots,k_n'$ with $g\in DK$ for some $D$ that does not divide $\gcd(c_1,\ldots,c_n)$, then $h'\notin DH$.  
\end{lemma}
\begin{proof}
It is enough to prove the claim for $D = p$, a prime.  

For $1\leq i\leq n$ let $Y_i = \langle k_1',\ldots,k_{i-1}',k_{i+1}',\ldots,k_n'\rangle_*$.  
Pick an arbitrary non-zero $h'\in H$.  

If $\outr \type K\wedge \type H = \type \Z$ then, for each $1\leq i \leq n$, the elements $h'\in H$ and $k_i' + Y_i \in K/Y_i$ have a greatest common divisor, that is, an integer $d_i$ that is maximal with the property that $h'\in d_iH$ and $k_i'+Y_i\in d_iK/Y_i$.  
By replacing $h'$ if necessary, we may suppose that the common divisor is $d_i = 1$ for all $1\leq i\leq n$.  

Pick an integer combination $g_0 := c_1k_1'+\cdots + c_nk_n'$ with $p\nmid \gcd (c_1,\ldots,$ $c_n)$, and suppose that $g_0 = pg_1\in pK$.  
This means that $c_ik_i'+Y_i = g_0+Y_i \in pK/Y_i$, and, since $p\nmid \gcd(c_1,\ldots,c_n)$, there is some $1\leq j\leq n$ such that $p\nmid c_j$.  
But then, applying Lemma \ref{LEM:euclidean} to the elements $g_1+Y_j$ and $k_j'+Y_j$ in the rank-one group $K/Y_j$, we obtain $g_2\in K$ with $g_1+Y_j = m_1g_2+Y_j$ and $k_j'+Y_j = m_2g_2+Y_j$.  
Thus 
\begin{align*}
c_jk_j' + Y_j & = pg_1 + Y_j \\
c_jm_2g_2 + Y_j & = pm_1g_2 + Y_j, 
\end{align*}
and therefore, since $K/Y_j$ is torsion-free, $c_jm_2  = pm_1$.  
Since $p\nmid c_j$, we must have $p\mid m_2$.  
But then $k_j+Y_j = m_2g_2+Y_j\in p K/Y_j$, so we conclude that $h'\notin pH$.  
\end{proof}

Using this property we can establish a necessary condition for $G$ to be ultrasimplicial.  

\begin{prop}\label{PROP:necessary}
If $\outr\type K\wedge \type H = \type \Z$, then $G$ is not ultrasimplicial.  
\end{prop}
\begin{proof}
Choose a maximal independet set $\{ k_1',\ldots,k_n'\}$ in $K$, and $h'\in H$ satisfying the conclusion of Lemma \ref{LEM:outertype}.  
By taking $-h'$ if necessary, we may suppose that $h'\in H^+$.  

Suppose that $G$ is ultrasimplicial.  
We will arrive at a contradiction by showing that $h'$ does not satisfy the conclusion of Lemma \ref{LEM:outertype}.  

Consider the $2n+1$ positive elements 
\begin{equation}\label{EQ:elements}
 (h',0),(h',k_1'),\ldots,(h',k_n'),(h',-k_1'),\ldots, (h',-k_n') 
\end{equation} 
in $G$.  
Since $G$ is ultrasimplicial we can find a set $\basisb$ of $n+1$ independent positive elements $\basisb = \{ (h_0,k_0),\ldots,(h_n,k_n)\}$ in $G$ such that the elements \ref{EQ:elements}  can be expressed as non-negative integer combinations of elements of $\basisb$.  

It will be convenient to write these non-negative integer combinations as a system of equations, that is, as a matrix equation.  
For this it will be necessary to introduce some basis elements.  
Use Lemma \ref{LEM:euclidean} to find $\overline{h}\in H$ with $\langle h_0,\ldots,h_n\rangle $ $ = \langle \overline{h}\rangle$ and $\overline{k_1},\ldots,\overline{k_n}\in K$ with $\langle k_0,\ldots,k_n\rangle $ $ = \langle \overline{k_1},\ldots,\overline{k_n}\rangle$.  
As usual, we may suppose that $\overline{h}\in H^+$.  

Say that 
\begin{align*}
k_i' & = \sum_{j=1}^n b_{ij}'\overline{k_j}, \quad  1\leq i \leq n,\\
k_i & = \sum_{j=1}^n b_{ij}\overline{k_j}, \quad 0\leq i\leq n,\\
h' & = a'\overline{h},\\
h_i & = a_i\overline{h}, 0\leq i \leq n,
\end{align*}
where $b_{ij},b_{ij}'\in\Z$ and $a_i,a'\in \Z_{\geq 0}$.  
Then to say that the elements \ref{EQ:elements}  can be expressed as non-negative integer combinations of elements of $\basisb$ means that there exist non-negative integer coefficients $c_{ij}$ with $0\leq i\leq 2n$ and $0\leq j\leq n$ such that the following matrix equation is satisfied.  
\begin{align}
CA & = %
\bigmatrix{a'}%
,\label{EQ:main} \\
\intertext{where}
\label{EQ:C} C & = \left[ \begin{array}{ccc}
c_{00} & \cdots & c_{0n} \\
 & & \\
\vdots & & \vdots \\
 & & \\
c_{2n+1,0} & \cdots & c_{2n+1,2n+1} 
\end{array}\right], \\
\label{EQ:A} A & =  \left[ \begin{array}{cccc}
a_0 & b_{01} & \cdots & b_{0n} \\
\vdots & \vdots & & \vdots \\
a_n & b_{n1} & \cdots & b_{nn} 
\end{array}\right],\\
\label{EQ:B-prime} \text{and }B' & = \left[ \begin{array}{ccc}
b_{11}' & \cdots & b_{1n}' \\
\vdots & & \vdots \\
b_{n1}' & \cdots & b_{nn}' 
\end{array}\right].
\end{align}

We will exploit the fact that the matrix $A$ in Equation \ref{EQ:A} has the property that $|\det(A)|>1$; a proof of this appears in Lemma \ref{LEM:determinant}, below.  

The rest of the proof works by showing that $D$ can be used to produce a contradiction with the conclusion of Lemma \ref{LEM:outertype}.  
Specifically, the following two statements are true.
\begin{enumerate}
\item  $h'\in DH$.
\item  There is an integer combination $g := q_1k_1'+\cdots q_nk_n'$ and a prime $p$ dividing $D$ such that $p\nmid \gcd(q_1,$ $\ldots,q_n)$ and $g\in pK$.  
\end{enumerate}

To prove these statements, let us rearrange Equation \ref{EQ:main} to obtain
\begin{align}
C & = \frac{1}{D} %
\bigmatrix{a'}%
\adj(A),\label{EQ:adjugate} 
\end{align}
where $\adj(A)$ denotes the adjugate of $A$---that is, the transpose of its matrix of cofactors.  

Let us first prove statement 1.  
Let $A(i,j)$ denote the $(i,j)$th cofactor of $A$, and hence, the $(j,i)$th entry of $\adj(A)$.  
Let $B$ denote the $(n+1)\times n$ matrix obtained by removing the first column of $A$.  
For $0\leq j \leq n$ the determinant of the matrix obtained by removing the $j$th row of $B$ is equal to $A(j,0)$.  
Then equating entries in the first row of Equation \ref{EQ:adjugate} implies that $c_{0j} = \frac{1}{D}a'A(j,0)$.  
Since each $c_{0j}$ is an integer, this means that $D\mid a'A(j,0)$ for all $j$.  

Choose a prime $p$ dividing $D$.  
If $p\mid A(j,0)$ for all $j$, then by the Cauchy-Binet formula \cite[Section 4.6]{BW:linear-algebra} the rows of $B$ span a sublattice of $\Z^n$ of index divisible by $p$.  
This means that the group elements $k_0,\ldots,k_n$ generate a subgroup of $\langle \overline{k_1},\ldots,\overline{k_n}\rangle$ of index divisible by $p$, which contradicts the choice of $\overline{k_1},\ldots,\overline{k_n}$.  
Therefore $p$ does not divide all $A(j,0)$, so it divides $a'$.  
Indeed, $p^r$ divides $a'$, where $p^r$ is the maximal power of $p$ dividing $D$.  
Therefore $D\mid a'$ and statement 1 is true.  

To prove statement 2, let ${B^j}'$ denote the $j$th column of $B'$ for $1\leq j \leq n$.  
Then Equation \ref{EQ:adjugate} implies in particular that, for each $1\leq i\leq n$, $D$ divides each entry of the integer combination
\begin{align*}
\sum_{j=1}^n {B^j}' A(i,j).
\end{align*}

Let us show that it is not possible for $D$ to divide each $A(i,j)$ with $j\geq 1$.  
Indeed, if $D$ does divide each $A(i,j)$ with $j\geq 1$, then $D$ divides every entry of every row of $\adj(A)$ except the first row.  
Then we can rewrite the scalar $1/D$ in Equation \ref{EQ:adjugate} as a product of two diagonal matrices as follows.  
\begin{align*}
C & = \bigmatrix{a'}%
\left[ \begin{array}{cccc}
\frac{1}{D} & & & \\
& 1 & & \\
& & \ddots & \\
& & & 1 
\end{array}\right]%
\left[ \begin{array}{cccc}
1 & & & \\
& \frac{1}{D} & & \\
& & \ddots & \\
& & & \frac{1}{D} 
\end{array}\right]%
\adj(A).
\end{align*}

By statement 1, the product of the first two matrices on the right hand side of this equation is an integer matrix.  
The product of the second two matrices is an integer matrix, and since $\det(\adj(A)) = D^{n}$, it is in fact unimodular.  

Let $A_0$ denote the inverse of the product of these last two matrices.  
The entries of $A_0$ have the same signs as the corresponding entries of $A$, so they are positive in the first column.  
Therefore we can rearrange Equation \ref{EQ:adjugate} again to obtain
\begin{align*}
CA_0 & = \bigmatrix{a'/D}%
\end{align*}
But this yields a contradiction with Lemma \ref{LEM:determinant}.  
Specifically, we have expressed $(\frac{a'}{D} h',0),(\frac{a'}{D}h',k_1'),\ldots,(\frac{a'}{D}h',k_n'),$ $(\frac{a'}{D}h',-k_1'),\ldots,(\frac{a'}{D}h',-k_n')$ as non-negative integer combinations of a maximal independent set in $G^+$, the elements of which are in turn expressed via the rows of $A_0$ as integer combinations of $(\overline{h},0),(0,\overline{k_1}),\ldots,(0,\overline{k_n})$.  
Lemma \ref{LEM:determinant} says that $A_0$ has determinant greater than $1$ in modulus, but we have already established that it is unimodular.  
This is a contradiction; therefore $D$ does not divide all $A(i,j)$ with $j\geq 1$.  

Then there is some index $i$ and some prime $p$ dividing $D$ such that $p$ divides every entry of $\sum_{j=1}^nA(i,j){B^j}'$, but $p\nmid \gcd(A(i,1),\ldots,A(i,n))$.  
This is almost what we want, but it is expressed in terms of columns of $B'$ rather than rows.  
Lemma \ref{LEM:row-column}, proved below, takes care of this.  

Let $B_j'$ denote the $j$th row of $B'$.  
Then by Lemma \ref{LEM:row-column}, there is some combination $\sum_{j=1}^n q_jB_j'$ that has every entry divisible by $p$, but $p\nmid\gcd(q_1,\ldots,q_n)$.  
This means that $\sum_{j=1}^n q_jk_j'\in pK$, and so statement 2 is true and the result follows.  
\end{proof}

The following two lemmas are used in the proof of Proposition \ref{PROP:necessary}.  

\begin{lemma}\label{LEM:determinant}
Pick a maximal independent set $k_1',\ldots,k_n'\in K$ and $h'\in H^+$, and suppose that $(h',0),(h',k_1'),\ldots,(h',k_n'),(h',-k_1'),\ldots,(h',-k_n')$ can be written as non-negative integer combinations of some independent elements $(h_0,k_0),\ldots,(h_n,k_n)\in G^+$.  
Let $\overline{h}\in H^+$ be such that $\langle h_0,\ldots, h_n\rangle$ $= \langle \overline{h}\rangle$, and let $\overline{k_1},\ldots,\overline{k_n}\in K$ be such that $\langle k_0,\ldots,k_n\rangle = \langle \overline{k_1},\ldots,\overline{k_n}\rangle$.  

Then the matrix $A$ from Equation \ref{EQ:A} has determinant greater than $1$ in modulus.  
\end{lemma}

\begin{proof}
The proof refers to the matrices in Equations \ref{EQ:main}, \ref{EQ:C}, \ref{EQ:A}, and \ref{EQ:B-prime}.  

Note that $\sum_{j=0}^nc_{0j}a_j = a'\in\N$, and because each $a_j\geq 1$ we can conclude that $a'>N := \sum_{j=0}^n c_{0j}$.  
Then $\sum_{j=0}^n \frac{c_{0j}}{N}a_j = \frac{a'}{N} > 1$, and the coefficients $\frac{c_{0j}}{N}$ all lie between $0$ and $1$.  
Moreover, $\sum_{j=0}^n \frac{c_{0j}}{N}b_{ji} = \frac{1}{N}\sum_{j=0}^nc_{0j}b_{ji} = 0$.  
Therefore the rows of $A$ span a parallelepiped that contains the row vector $(1,0,\ldots,0)$.  

In fact, let us show that this vector lies in the interior of that parallelepiped; this will imply that $|\det(A)|>1$.  
To show that $(1,0,\ldots,0)$ lies in the interior of the parallelepiped spanned by the rows of $A$, it is necessary to show that, for all $j$, $c_{0j}\neq 0$. 

For $1\leq j\leq n$ let $B_j'$ denote the $j$th row of $B'$ (Equation \ref{EQ:B-prime}), and for $0\leq j\leq n$ let $B_j$ denote the row vector obtained by dropping the first entry from the $j$th row of $A$: $B_j = (b_{j1},\ldots,b_{jn})$.  
Then for $1\leq i\leq n$,
\begin{align*}
\sum_{j=0}^nc_{ij}B_{j} & = B_{j}' & \text{and}\quad  \sum_{j=0}^nc_{n+i,j}B_{j} & = -B_{j}',
\end{align*}
therefore
\begin{align*}
\sum_{j=0}^n(c_{ij}+c_{n+i,j})B_{j} & = 0 = \sum_{j=1}^n c_{0j}B_j.
\end{align*}

These are both relations between the $n+1$ vectors $B_j, 0\leq j\leq n$.  
But these $n+1$ vectors must span the row space of $B'$, which we have chosen to be $n$-dimensional; therefore there is a unique relation between them up to scalar multiplication.  
Thus if $c_{0j} = 0$ then $c_{ij}+c_{n+i,j} = 0$, which implies that $c_{ij}$ and $c_{n+i,j}$ are both $0$ since they are both non-negative.  
This means that if any $c_{0j} = 0$, then the entire $j$th column of $C$ is $0$, which implies that $C$ has rank less than $n+1$.  
Since the right hand side of Equation \ref{EQ:main} was chosen to have rank $n+1$, this is not possible.  
Therefore for all $j$, $c_{0j}\neq 0$, and so $|\det (A)| > 1$.  
\end{proof}

\begin{lemma}\label{LEM:row-column}
Let $M$ be an $n\times n$ integer matrix and let $p$ be a prime.  
Let $M_j$ denote the $j$th row of $M$ and let $M^j$ denote its $j$th column.  
Then the following are equivalent.  
\begin{enumerate}
\item  There is an integer combination $\sum_{j=1}^n r_jM_j$ such that every entry is divisible by $p$ but $p\nmid \gcd(r_1,\ldots,r_n)$.  
\item  $p\mid \det(M)$.  
\item  There is an integer combination $\sum_{j=1}^n s_jM^j$ such that every entry is divisible by $p$ but $p\nmid \gcd(s_1,\ldots,s_n)$.  
\end{enumerate}
\end{lemma}
\begin{proof}
Let $R$ denote the row vector with entries $r_j$ and $A$ denote the integer combination in (1).  
Then the fact that (1) implies (2) is easily verified by rewriting the matrix equation $RM = A$ using the adjugate formula: $R = 1/(\det(M)) A\adj(M)$ and observing that $\adj(M)$ is an integer matrix.  

The fact that (2) implies (1) can be seen by noting that the rows of $M$ generate (as a group) a proper sublattice of $\Z^n$.  
If $A$ is a row vector in $\Z^n\setminus \langle M_1,\ldots,M_n\rangle$, then $\det(M)A$ is in $\langle M_1,\ldots,M_n\rangle$, so $\det(M)A = RM$ for some integer row vector $R$.  
If the entries of $R$ were all divisible by $\det(M)$, then $A$ would lie in $\langle M_1,\ldots,M_n\rangle$.  

The equivalence of (2) and (3) follows by applying the same reasoning to the matrix $M^t$.  
\end{proof}

In certain special cases the conditions in Propositions \ref{PROP:sufficient} and \ref{PROP:necessary} can be expressed more simply.  

\begin{cor}\label{COR:direct-sum}
If $K$ is a direct sum of rank-one groups, then $G$ is ultrasimplicial if and only if $\inner \type G \neq \inner\type \Z$.  
\end{cor}
\begin{proof}
If $K$ is a direct sum of rank-one groups, then $\inner\type K = \outr\type K$, so Propositions \ref{PROP:sufficient} and \ref{PROP:necessary} combine to say that $G$ is ultrasimplicial if and only if $\inner\type K \wedge \type H \neq \type \Z$.  
And $\inner\type K\wedge \type H = \inner\type G$.  
\end{proof}

Corollary \ref{COR:direct-sum} yields the following example as a special case.  

\begin{example}\label{EX:direct-sum}
Let $m_0,\ldots,m_n$ be non-zero integers, let $H = \Z\big[ \frac{1}{m_0}\big]$, and let $K = \Z\big[ \frac{1}{m_1}\big] \oplus \cdots \oplus \Z\big[ \frac{1}{m_n}\big]$.  
Then $G$ is utrasimplicial if and only if $\gcd(m_0,\ldots,m_n) > 1$.  
\end{example}

If $H$ is cyclic, then $\outr\type K \wedge \type H = \type \Z$ for any choice of $K$, so $G$ is not ultrasimplicial.  
But in this case we can say even more: $G$ is not a dimension group.  
Elliott observed in \cite{E:example} that every dimension group $A$ satisfies the Riesz interpolation property: if $x_1,x_2,y_1,y_2\in A$ satisfy $x_i\leq y_j$ for all $i,j$, then there is an element $z$, called an \defemph{interpolant}, such that $x_i\leq z\leq y_j$.  
In \cite{EHS}, Effros, Handelman, and Shen proved a converse to this.  
Specifically, every countable ordered group with Riesz interpolation is a dimension group.  
It is not difficult to check that, if $H$ is cyclic and $K$ is non-trivial, then $G$ does not satisfy the Riesz property, and hence is not a dimension group.

The next example is a group that is does not satisfy the sufficient condition of Proposition \ref{PROP:sufficient} nor the necessary condition of Proposition \ref{PROP:necessary}.  
\begin{example}\cite[Example 2.8]{A:book}.\label{EX:strange}  
Consider the group $V = \Q^n$.  
Let $T := \{ (r_1,\ldots,r_n)\in V \ | \ r_i\in\Z \text{ and } \gcd(r_1,\ldots,r_n) = 1\}$.  
Enumerate the elements of $T$: $T = \{ t_1,t_2,\ldots\}$ and write the set $\Pi$ of primes as an infinite disjoint union of infinite sets $\Pi = \bigcup S_i$.  
Then let $K$ be the subgroup of $V$ generated by all elements of the form $\frac{1}{p_i}t_i$, where $p_i\in S_i$ and let $H$ be a rank-one group with a type that is non-zero at infinitely many primes $p$.  

Then $\inner\type K = \type\Z$ and $\outr\type K = [(m_p)_{p\in\Pi}]$, where $m_p = 1$ for all $p$, so neither Proposition \ref{PROP:sufficient} nor Proposition \ref{PROP:necessary} applies, and it is not clear if the group $G$ is ultrasimplicial or not.  
\end{example}

This final example demonstrates that we cannot remove from Proposition \ref{PROP:necessary} the condition that the exact sequence associated to the unique state split.  

\begin{example}\label{EX:non-splitting}
Consider the inductive limit of the diagram 

\centerline{
\xymatrix{
\Z^2 \ar@{>}^{A}[r] & \Z^2 \ar@{>}^{A}[r] & \Z^2 \ar@{>}^{A}[r] & \cdots 
}
}

in which $A$ has matrix representation $A = \left[ \begin{array}{cc}2 & 1 \\ 1 & 2 \end{array}\right]$.  
The matrix $A$ has right eigenvectors $(1,1)^t$ and $(1,-1)^t$ with eigenvalues $3$ and $1$ respectively.  
But these eigenvectors generate only a sublattice of $\Z^2$ with index $2$.  
Thus the limit is isomorphic to the subgroup $G$ of $\Q^2$ that is generated by the elements $\{ (1/3^n,0) \ | \ n\in \Z\} \cup \{ (0,1) \} \cup \{ (1/2,1/2)\}$ and that has positive cone $G^+ = \{ (a,b)\in G \ | \ a > 0\} \cup \{ 0\}$.  
This dimension group $G$ is clearly ultrasimplicial as $A$ is non-singular.  

The group in Example \ref{EX:elliott} is a finite-index order subgroup of $G$, and, indeed, the two are isomorphic as groups.  
The difference here is that the exact sequence associated to the unique state fails to split.  
This shows that it is not just the group structure that determines ultrasimpliciality in these examples, but also the order structure arising from the state.  
\end{example}

\begin{example}\label{EX:finite-index-subgroup}
More generally, if $H = \Z\big[\frac{1}{m_0}\big]$ and $K = \Z\big[\frac{1}{m_1}\big]$ with $m_0>m_1$, then $G$ is isomorphic to a finite-index order subgroup of an ultrasimplicial group.  
We can represent such a group as an inductive limit using the diagram 

\centerline{
\xymatrix{
\Z^2 \ar@{>}^{A}[r] & \Z^2 \ar@{>}^{A}[r] & \Z^2 \ar@{>}^{A}[r] & \cdots 
}
}

with
\begin{align*}
A & = \left[ \begin{array}{cc}
\lfloor \frac{m_0+m_1}{2}\rfloor & \lceil \frac{m_0-m_1}{2}\rceil \\
\lfloor \frac{m_0-m_1}{2}\rfloor & \lceil \frac{m_0+m_1}{2}\rceil 
\end{array}\right],
\end{align*}
where $\lfloor \cdot\rfloor$ and $\lceil \cdot \rceil$ denote the floor and ceiling functions respectively.  

Note that the requirement that $m_0$ be greater than $m_1$ is easy to circumvent because $\Z\big[\frac{1}{m_0}\big] = \Z\big[\frac{1}{m_0^n}\big]$ for any $n\in\N$.  
\end{example}

\bibliographystyle{abbrv}
\bibliography{ultrasimplicial}

\begin{thebibliography}{10}

\bibitem{A:book}
D.~M. Arnold.
\newblock {\em Finite rank torsion free abelian groups and rings}, volume 931
  of {\em Lecture Notes in Mathematics}.
\newblock Springer-Verlag, Berlin-New York, 1982.

\bibitem{BW:linear-algebra}
J.~G. Broida and S.~G. Williamson.
\newblock {\em A comprehensive introduction to linear algebra}.
\newblock Addison-Wesley Publishing Company, Advanced Book Program, Redwood
  City, CA, 1989.

\bibitem{EHS}
E.~G. Effros, D.~E. Handelman, and C.~L. Shen.
\newblock Dimension groups and their affine representations.
\newblock {\em Amer. J. Math.}, 102(2):385--407, 1980.

\bibitem{ES:multidimensional}
E.~G. Effros and C.~L. Shen.
\newblock Dimension groups and finite difference equations.
\newblock {\em J. Operator Theory}, 2(2):215--231, 1979.

\bibitem{ES:rank-two}
E.~G. Effros and C.~L. Shen.
\newblock Approximately finite {$C^{\ast} $}-algebras and continued fractions.
\newblock {\em Indiana Univ. Math. J.}, 29(2):191--204, 1980.

\bibitem{E:example}
G.~A. Elliott.
\newblock On totally ordered groups, and {$K_{0}$}.
\newblock In {\em Ring theory ({P}roc. {C}onf., {U}niv. {W}aterloo, {W}aterloo,
  1978)}, volume 734 of {\em Lecture Notes in Math.}, pages 1--49. Springer,
  Berlin, 1979.

\bibitem{G:book}
K.~R. Goodearl.
\newblock {\em Partially ordered abelian groups with interpolation}, volume~20
  of {\em Mathematical Surveys and Monographs}.
\newblock American Mathematical Society, Providence, RI, 1986.

\bibitem{H:ultrasimplicial}
D.~Handelman.
\newblock Ultrasimplicial dimension groups.
\newblock {\em Arch. Math. (Basel)}, 40(2):109--115, 1983.

\bibitem{H:almost-ultrasimplicial}
D.~Handelman.
\newblock Equal column sum and equal row sum dimension group realizations.
\newblock {\em Pre-print}, 2013.
\newblock http://arxiv.org/abs/1301.2799.

\bibitem{M:lattice-ordered}
V.~Marra.
\newblock Every abelian {$l$}-group is ultrasimplicial.
\newblock {\em J. Algebra}, 225(2):872--884, 2000.

\bibitem{R:algorithm}
N.~Riedel.
\newblock Classification of dimension groups and iterating systems.
\newblock {\em Math. Scand.}, 48(2):226--234, 1981.

\bibitem{R:counterexample}
N.~Riedel.
\newblock A counterexample to the unimodular conjecture on finitely generated
  dimension groups.
\newblock {\em Proc. Amer. Math. Soc.}, 83(1):11--15, 1981.

\bibitem{T:vector-spaces}
A.~Tikuisis.
\newblock Finite dimensional ordered vector spaces with {R}iesz interpolation
  and {E}ffros-{S}hen's unimodularity conjecture.
\newblock {\em Pre-print}, 2011.
\newblock http://arxiv.org/abs/1110.6851.

\end{thebibliography}

\end{document}